\documentclass[reqno,12pt]{amsart}
\usepackage{amscd,amsmath,amsthm,amsfonts,latexsym,amssymb, mathrsfs,stmaryrd,graphics,pst-all,amscd,xypic,tkz-euclide,yfonts,mathtools,tikz,graphicx,wrapfig}
\usepackage{tikz,tkz-graph,graphicx,wrapfig}
\usetikzlibrary{decorations.markings}
\tikzstyle{angle}=[circle,draw,inner sep=0pt, minimum size=6pt]
\scrollmode
\swapnumbers
\theoremstyle{plain}
\newtheorem{theorem}{Theorem}[section]
\newtheorem{lemma}[theorem]{Lemma}

\theoremstyle{definition}

\theoremstyle{remark}
\newtheorem{remark}[theorem]{Remark}

\numberwithin{equation}{section}

\newcommand{\forget}[1]{}

\begin{document}

\title[Prinipally-INJECTIVITY LEAVITT
PATH ALGEBRAS OVER ARBITRARY GRAPHS]{ Principally-injective Leavitt path algebras over arbitrary graphs}

\author{Soumitra Das and Ardeline M. Buhphang*}

\address{S. Das, Department of Mathematics, North-Eastern Hill University,
	Permanent Campus, Shillong-793022, Meghalaya, India.}
\email{soumitrad330@gmail.com}

\address{A. M. Buhphang, Department of Mathematics, North-Eastern Hill University,
	Permanent Campus, Shillong-793022, Meghalaya, India.}

\email{ardeline17@gmail.com}

\begin{abstract} A ring $R$ is called right {\it principally-injective} if every $R$-homomorphism $f : aR \longrightarrow R$, $a \in R$, extends to $R$, or equivalently, if every system of equations $xa = b \ (a,b \in R)$ is solvable in $R$. In this paper we show that for any arbitrary graph $E$ and for a field $K$, principally-injective conditions for the Leavitt path algebra $L_{K}(E)$ are equivalent to that the graph $E$ being acyclic. We also show that the principally-injective Leavitt path algebras are precisely the von Neumann regular Leavitt path algebras. 
\end{abstract}

\thanks{*Corresponding author}

\maketitle

\section{Introduction} \label{S:intro} 

All the rings that we consider here are assumed to be associative with local units (such as the Leavitt path algebras).\\ One of the fascinating direction to study in Leavitt path algebras  is the characterization of the ring-theoretic properties of a Leavitt path algebra $L_{K}(E)$  in terms of the graph-theoretic properties of the graph $E$ (see chapter 4 \cite{Abrams}). This motivates us to study Principally-injective Leavitt path algebras.\\ Recall that a ring
$R$ is {\it locally unital} if for each finite set $F$ of elements of $R$, there is an idempotent $u$ (i.e $u^{2}=u \in R$) such that $ua = au = a$ for all $a \in F$. The set of all such idempotents $u$ is said to be a set of local units. A ring $R$ is said to be (von Neumann) regular if each $a \in R$ satisfies $a \in aRa$. The von Neumann regular Leavitt path algebras $L_{K}(E)$ of arbitrary graphs $E$ over a field $K$ were characterized in \cite{Abrams1} in terms of the graphical properties of $E$, namely, the graphs $E$ must have no cycles. \\For a subset $X$ of a ring $R$ (not necessarily unital), the set\\ $r_{R}(X) =\lbrace t \in R: xt=0 \ \forall x \in X\rbrace$ is right annihilator of $X$. The left annihilator $l_{R}(X)$ is also defined in a similar fashion for $X \subseteq R$. It is straightforward to check that $r_{R}(X)$ is a right and $l_{R}(X)$ is a left ideal of $R$. A ring $R$ is called right {\it principally injective (P-injective)} if every $R$-homomorphism $f : aR \longrightarrow R$, $a \in R$, extends to $g : R \longrightarrow R$ or equivalently, if every system of equations $xa = b \ (a,b \in R)$ has a solution $x$ in $R$ (see  Proposition 3.17\cite{Lam}). Thus every right self-injective ring is right P-injective. We shall see in Lemma \ref{ecpinj} that the following are equivalent for a (locally unital) ring $R$ (i) $R$ is right P-injective (ii) $l_{R}r_{R}(a) = Ra$ for all $a \in R$. We note in Lemma \ref{regular} that every (locally unital) regular ring $R$ is both right and left P-injective.  As consequences we will prove in Theorem \ref{thm} that every regular Leavitt path algebra $L_{K}(E)$ is regular iff $L_{K}(E)$ is $P$-injective iff the graph $E$ contains no cycle. 
\\For the other definitions in this note, we refer to \cite{Nicholson}, \cite{Lam}and \cite{Aranda}.\\

\section{Preliminaries}\label{S:pre}
We recall the fundamental terminology for our note which can be find in the text \cite{Abrams}. For the sake
of completeness, we shall outline some of the concepts and results that we will be using.
A (directed) graph $E = (E^{0};E^{1}; r; s)$ consists of two sets $E^{0}$ and $E^{1}$ together with maps $r, s : E^{1} \longrightarrow E^{0}$. The elements of $E^{0}$ are called $vertices$ and the elements
of $E^{1}$ $edges$. For each $e \in E^{1}$, $r(e)$ is the range of $e$ and $s(e)$ is the source of $e$. If $r(e) = v$
and $s(e) = w$, then we say that $v$ emits $e$ and that $w$ receives $e$. A vertex $v$ is called a $sink$ if it emits no edges and a vertex $v$ is called a $regular$ vertex if it emits a non-empty finite set of edges. An $infinite\ emitter$ is a vertex which emits infinitely many edges. For each $e \in E^{1}$, we call $e^{*}$ a ghost edge. We let $r(e^{*})$ denote $s(e)$, and we let $s(e^{*})$ denote $r(e)$. A path $\mu$ of length $n \gneq 0$ is a finite sequence of edges $\mu = e_{1} e_{2}\cdot  \cdot \cdot e_{n}$ with $r(e_{i}) = s(e_{i+1})$ for all $i = 1,\cdot \cdot \cdot, n-1$. In this case $\mu^{*} = e_{n}^{*} \cdot \cdot \cdot e_{2}^{*}e_{1}^{*}$ is the corresponding ghost path. A vertex is considered a path of length $0$. For a vertex $v$, we define $v^{*} = v$. The set of all vertices on the path $\mu$ is denoted by $\mu^{0}$. The set of all paths in $E$ is denoted by Path$(E)$.
A path $\mu = e_{1} e_{2}...... e_{n}$ in $E$ is $closed$ if $r(e_{n}) = s(e_{n})$, in which case $\mu$ is said to be based at the vertex $s(e_{1})$. A closed path $\mu$ as above is called $simple$ provided it does not pass through its base more than once, i.e., $s(e_{i}) \neq  s(e_{1})$ for all $i = 2, ..., n$. The closed path $\mu$ is called a $cycle$ if it does not pass through any of its vertices twice, that is, if $s(e_{i}) \neq s(e_{j})$ for every $i \neq j$.\\
Given an arbitrary graph $E$ and a field $K$, the Leavitt path algebra $L_{K}(E)$ is defined to be the $K$-algebra generated by a set $\lbrace v : v \in E^{0} \rbrace$ of pair-wise orthogonal idempotents
together with a set of variables $\lbrace e, e^{*} : e \in E^{1} \rbrace$ which satisfy the following conditions:\\
(1) $s(e)e = e = er(e)$ for all $e \in E^{1}$.\\
(2) $r(e)e^{*} = e^{*} = e^{*}s(e)$ for all $e \in E^{1}$.\\
(3) (The CK-1 relations) For all $e, f \in E^{1},\ e^{*}e = r(e)$ and\\ $e^{*}f = 0$ if $e \neq f$.\\
(4) (The CK-2 relations) For every regular vertex $v \in E^{0}$,\\
$v = \sum\limits_{e \in E^{1},\ s(e)=v} ee^{*}$.\\

A useful observation is that every element $a$ of $L_{K}(E)$ can be written in the form $a =\sum_{i=1}^{n}k_{i}\alpha_{i}\beta_{i}^{*} $, where $k_{i} \in K$, $\alpha_{i}$, $\beta_{i}$ are paths in $E$ and $n$ is a suitable integer.

We  mention  two basic examples:\\

$(1)$ If $E$ is the graph having one vertex and a single loop\\
\begin{tikzpicture}[->,>=stealth',shorten >=2pt,auto,node distance=4cm,thick]
\node[text width=3cm, anchor=south,left] at (-.001,0){};

\node at (0,0) (1) {$\bullet_{v}$};

\path
(1) edge [loop above] node {$c$} (1);
  
\end{tikzpicture}

then $L_{K}(E)\cong K[x,x^{-1}]$, the Laurent polynomial $K$-algebra via $v \longmapsto 1$, $c\longmapsto x$ and $c^{*}\longmapsto x^{-1}$.

$(2)$ If $E$ is the oriented $n$-line graph having $n$ vertices and $n-1$ edges

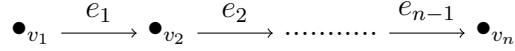
\begin{tikzpicture}
\node[text width=3cm, anchor=south,left] at (-.001,0){};
\node at (0,0) (1) {$\bullet_{v_{1}}$};
\node[right =1cm of 1] (2) {$\bullet_{v_{2}}$};
\node[right =1cm of 2] (3) {$...........$};
\node[right =1cm of 3] (4) {$\bullet_{v_{n}}$};
\draw[->](1) to node[above] {$e_{1}$} (2);
\draw[->](2) to node[above] {$e_{2}$} (3);
\draw[->](3) to node[above] {$e_{n-1}$} (4);
\end{tikzpicture}\\
then $L_{K}(E)\cong M_{n}(K)$,  via $v_{i}\longmapsto f_{i,i}$, $e_{i}\longmapsto f_{i,i+1}$ and $e_{i}^{*} \longmapsto f_{i+1,i}$, where $\lbrace f_{i,j} : 1 \leq i,j\leq n \rbrace$ denotes the standard matrix units in $M_{n}(K)$.
\\\\

\section{Results} 

We shall need the following lemmas.
\begin{lemma} (cf. Lemma 5.1 \cite{Nicholson})\label{ecpinj} The following conditions are equivalent for a locally unital ring $R$:\\ (i) $R$ is right P-injective.\\
(ii) $l_{R}r_{R}(a) = Ra$ for all $a \in R$.
\end{lemma}
\begin{proof} $(i) \implies (ii)$ For any $z \in r_{R}(a)$, $az= 0$. This implies that $a \in l_{R}(z), \forall z \in r_{R}(a)$, yielding $Ra \subseteq l_{R}r_{R}(a)$. Let $x \in l_{R}r_{R}(a)$. Then, $r_{R}(a)\subseteq r_{R}(x)$. Define $f : aR \rightarrow R$ by $f(at)=xt$. This is well-defined as $at = at'$ implies that
	$a(t-t')= 0$, that is $(t- t') \in  r_{R}(a) \subseteq r_{R}(x)$. So, $f(at) = xt= xt' = f(at')$. By (i), there exist $g: R \rightarrow R$ such that $f(a)= g(a)=  g(ua)= g(u)a \in Ra$, where $u$ is the local unit for $a$. Hence $l_{R}r_{R}(a) = Ra$.\\$(ii) \implies (i)$ Let $f : aR  \rightarrow R$, $a \in R$, be $R$-linear. Then $f(a) = d$, for
	some $d \in R$. We show that $d \in Ra$. Take $x \in r_{R}(a)$. Then $0 = f(ax) = f(a)x = dx$, so, $r_{R}(a) \subseteq r_{R}(d)$. This implies that $l_{R}r_{R}(d) \subseteq l_{R}r_{R}(a)$. So, $d \in l_{R}r_{R}(a)=Ra$, therefore, $d \in Ra$. Hence $f(a)=d=ca$ for some $c \in R$. Define $g : R \rightarrow R$ by $g(x)=cx$. Then $g$ is the required extension of $f$ on $R$.
\end{proof}
Before deriving the next lemma, we insert a remark here.
\begin{remark} For the graph $E=$ \begin{tikzpicture}[->,>=stealth',shorten >=2pt,auto,node distance=4cm,thick]
	\node[text width=3cm, anchor=south,left] at (-.001,0){};
	
	\node at (0,0) (1) {$\bullet_{v}$};
	
	\path
	(1) edge [loop above] node {$c$} (1);
\end{tikzpicture}
	
the corresponging Leavitt path algebra $R=L_{K}(E)$ is not right (left) P-injective can be seen as follows: $c^{*} \in l_{R}r_{R}(v-c)$ but $c^{*} \notin R(v-c)$. Thus $l_{R}r_{R}(v-c) \neq R(v-c)$.
\end{remark}

\begin{lemma}\label{regular} Let $R$ be a locally unital ring. If $R$ is (von Neumann) regular then $R$ is right (left) P-injective.
\end{lemma}
\begin{proof}Always, $Ra \subseteq l_{R}r_{R}(a)$ for any $a \in R$. To see $l_{R}r_{R}(a) \subseteq Ra$. Let us take $x \in l_{R}r_{R}(a)$, then $r_{R}(a)\subseteq r_{R}(x)$. Write $a=ara$ for some $r \in R$ and choose $v$ a local unit of $x$. Then $(v-ra)va \in r_{R}(a)\subseteq r_{R}(x)$. This shows that $x=xrava \in Ra$. The result now follows by Lemma \ref{ecpinj}.
\end{proof}

\indent Recall that a ring $R$ is said to be {\it semiprime} if, for every ideal $I$ of $R$, $I^{2}=0$ implies $I=0$.

\begin{lemma}(cf. Lemma 4.3.4 \cite{Abrams}) \label{pp} Let $R$ be a locally unital ring which is semiprime and right P-injective. Then for every idempotent $e \in R$, the corner $eRe$ is right P-injective.
\end{lemma}
\begin{proof}Write $S = eRe$ and let $x \in l_{S}r_{S}(a)$, where $x, a \in S$. Then $r_{S}(a) \subseteq r_{S}(x)$. By Lemma  \ref{ecpinj}, it suffices to show that $r_{R}(a) \subseteq r_{R}(x)$ (then $x \in Ra$, so $x=ex\in eRa=Sa$, as required). So let $y \in r_{R}(a)$, then $ay=0 \Rightarrow aey=0 \Rightarrow aeyke=0,\ \forall k \in R \Rightarrow  xeyke=0, \ \forall k \in R$ (since $r_{S}(a) \subseteq r_{S}(x)$). Thus, $xyRe=0$ and $exy=xy$. Now consider the two-sided ideal $RxyR$ of $R$, and note that $(RxyR)^{2} \subseteq RxyRxyR \subseteq RxyRexyR=\lbrace 0 \rbrace$. But $R$ is semiprime and hence $RxyR=\lbrace 0 \rbrace$ implying that $xy=0$ as $R$ has local units. This completes the proof.
	
\end{proof}
\indent Recall that a ring is a right {\it PP ring} if every principal right ideal is projective.\indent It is worth mentioning that a ring $R$ without identity may not be a projective $R$-module. But a Leavitt path algebra over an arbitrary graph is always projective as a module over itself (see Corollary 2.3 \cite{Aranda}).\\

\indent We use a part of Lemma 8 \cite{Hazrat} and note the following lemma.

\begin{lemma}(cf. Example 5.8. \cite{Nicholson})\label{von} Let $R$ be a locally unital ring which is semiprime. If $R$ is right P-injective, right PP ring  then $R$ is (von Neumann) regular.
\end{lemma}
\begin{proof}Let $a$ be any element in $R$ with $u$ a local unit of $a$ . Since $R$ is a right PP ring, $aR$ is projective and so the short exact sequence $0\rightarrow r_{R}(a) \rightarrow R \rightarrow aR \rightarrow 0$ splits. Writing $S=uRu$ and arguing as in Lemma 8 \cite{Hazrat} we get that the following short exact sequence $0\rightarrow r_{S}(a) \rightarrow S \rightarrow aS \rightarrow 0$ splits. Thus, $r_{S}(a)=eS$, where $e^{2}=e \in S$. Hence, (by Lemma  \ref{ecpinj}) $Sa=l_{S}r_{S}(a)=S(u-e)$. Now $Sa \oplus Se=S$ implies that $a \in aRa$.
\end{proof}

\indent Before stating our main theorem, we shall recall few more lemmas.

\begin{lemma}(see Proposition 2.3.1 \cite{Abrams})\label{semiprime} Let $E$ be an arbitrary graph and $K$ be any field. Then the Leavitt path algebra $L_{K}(E)$ is semiprime.
\end{lemma}

\begin{lemma}(see Theorem 3.7. \cite{Ara}) \label{onesided} Let $E$ be an arbitrary graph and $K$ be any field. Then every one-sided ideal of $L_{K}(E)$ is projective.
\end{lemma}

We are now in a position to show that every right (or left) P-injective Leavitt path algebras are (von Neumann) regular.
\begin{theorem}\label{thm} Let $E$ be an arbitrary graph and $K$ be any field. Then the following are equivalent.\\(1) $L_{K}(E)$ is right (left) P-injective\\ (2) $L_{K}(E)$ is von Neumann regular.\\
	(3) $E$ is acyclic (i.e contains no cycle).\\
	(4) $L_{K}(E)$ is locally $K$-matricial.\\ 
\end{theorem}
\begin{proof}$(1)\Rightarrow(2)$ Follows from Lemma  \ref{semiprime}, Lemma  \ref{onesided} and Lemma  \ref{von}. $(2) \Rightarrow (1)$ Follows from Lemma \ref{regular}. The equivalence of $(2)$, $(3)$ and $(4)$ can be seen in Theorem 3.4.1 \cite{Abrams}.

\end{proof}

\section*{Acknowledgements}
We would like to express our gratitude to Prof. Kulumani M. Rangaswamy and Prof. Gene Abrams for  fruitful conversations. Thanks are due to Prof. M. B. Rege for his comments and suggestions which significantly improved the exposition.

\end{document}